\documentclass[reqno]{amsart}
\usepackage{amsfonts}
\usepackage{amssymb, amsmath}
\usepackage{natbib}
\usepackage{lscape}
\usepackage{dsfont}
\usepackage{mathrsfs}
\usepackage{bbm}
\usepackage[pdftex,plainpages=false,colorlinks,hyperindex,bookmarksopen,linkcolor=red,citecolor=blue,urlcolor=blue]{hyperref}

\DeclareMathAlphabet{\mathpzc}{OT1}{pzc}{m}{it}
\bibpunct{[}{]}{;}{n}{,}{,}
\newtheorem{te}{Theorem}[section]

\newtheorem{os}[te]{Remark}
\newtheorem{prop}[te]{Proposition}

\newtheorem{ex}[te]{Example}
\numberwithin{equation}{section}
\allowdisplaybreaks

\begin{document}
\title[]{ Random time-change with inverses of multivariate subordinators:
governing equations and fractional dynamics}
\keywords{Random time-change, multivariate L\'evy processes, subordinators,
anomalous diffusion, continuous time random walks, fractional operators.\\
\\
* Dipartimento di Scienze Statistiche, Universit\`a di Roma Sapienza,
Piazzale Aldo Moro 5, 00185, Roma, Italy\\
** Dipartimento Matematica, Universit\`a di Roma Tor Vergata, viale della
Ricerca Scientifica 1, 00133, Roma, Italy }
\date{\today }
\author[]{Luisa Beghin*}
\author[]{Claudio Macci**}
\author[]{Costantino Ricciuti*}

\begin{abstract}
It is well-known that compositions of Markov processes with inverse
subordinators are governed by integro-differential equations of generalized
fractional type. This kind of processes are of wide interest in statistical
physics as they are connected to anomalous diffusions. In this paper we
consider a generalization; more precisely we mean componentwise compositions
of $\mathbb{R}^d$-valued Markov processes with the components of an
independent multivariate inverse subordinator. As a possible application, we
present a model of anomalous diffusion in anisotropic medium, which is
obtained as a weak limit of suitable continuous-time random walks.
\end{abstract}

\maketitle
\tableofcontents

{}

{}

\section{Introduction}

A multivariate subordinator is a $\mathbb{R}^d$-valued L\'evy process with
non-decreasing (possibly dependent) marginal components. In \cite{sato}, the
authors studied the subordinated process $(X_1(H_1(t)), \dots, X_d(H_d(t)))$%
, where $(X_1 (t),\dots, X_d (t))$ is a time-homogeneous $\mathbb{R}^d$%
-valued Markov process and $(H_1 (t),\dots, H_d (t))$ is an independent
multivariate subordinator. This generalized the well-known case of a common
random time $H(t)$ for all the components (for further developments consult
\cite{pedersen 1}, \cite{pedersen 2}, \cite{pedersen 3}). Time-changes of
Markov processes with multivariate subordinators also appear in some
references in the literature with applications to finance; see e.g. \cite%
{LucianoSemeraro}.

In the same spirit, this paper focuses on the time-changed process
\begin{align}
\bigl (X_1(L_1(t)), \dots, X_d(L_d(t)) \bigr), \qquad t\geq 0
\label{processo subordinato nostro}
\end{align}
where $(X_1(t), \dots, X_d(t))$ is a $\mathbb{R}^d$-valued Markov process,
while, for each $j=1, \dots, d$, $L_j(t)$ is the inverse (or
right-continuous hitting time) of $H_j(t)$, i.e.
\begin{align*}
L_j (t)=\inf \{x>0: H_j(x)>t\},
\end{align*}
which is assumed to be independent of $(X_1(t), \dots, X_d(t))$. Without
loss of generality, we assume that $(X_1(0), \dots, X_d(0))=(0,\dots,0)$
almost surely.

The literature inspiring our study is vast. Indeed many works, such as \cite%
{becker}, \cite{kolokoltsov}, \cite{meer libro}, \cite{meer spa}, \cite{meer
jap}, \cite{meer straka}, \cite{meer toaldo}, concern processes of the form
\begin{align}
(X_1(L(t)), \dots, X_d(L(t)), \qquad t\geq 0,
\label{processo subordinato classico}
\end{align}
where each component of the original process is time-changed by the same
random time $L (t)=\inf \{x>0: H(x)>t\}$, namely the inverse of the same
univariate subordinator.

There exists a well established theory for the stochastic process in (\ref%
{processo subordinato classico}). An interesting fact is that the density $%
p(x,t)$ of (\ref{processo subordinato classico}), $x\in \mathbb{R}^d$, is
governed by integro-differential equations of the form
\begin{align}
\mathcal{D}_t p(x,t)= \mathcal{G}p(x,t), \qquad x\neq 0,  \label{eq 1 intro}
\end{align}
where $\mathcal{D}_t$ is the operator defined by
\begin{align}
\mathcal{D}_t h(t):= \int _0^\infty (h(t)-h(t-\tau))\nu(d\tau)
\label{operatore univariato}
\end{align}
$\nu$ is the L\'evy measure of the underlying subordinator $H(t)$ and $%
\mathcal{G}$ is the dual to the generator of $(X_1 (t), \dots, X_d(t))$. The
operator $\mathcal{D}_t$ is called generalized fractional derivative
because, when $L(t)$ is an inverse stable subordinator of index $\alpha \in
(0,1)$, it reduces to the Caputo fractional derivative
\begin{align}
\frac{d^\alpha}{dt^\alpha} f(t)= \int _0^\infty \frac{d}{d\tau}f(\tau) \frac{%
(t-\tau)^{-\alpha -1}}{\Gamma (1-\alpha)}d\tau  \label{Caputo}
\end{align}
(see e.g. \cite{kilbas} p. 92 for details). Moreover, if $(X_1(t), \dots,
X_d(t))$ is a L\'evy process, then processes defined in (\ref{processo
subordinato classico}) can be seen as scaling limits of suitable continuous
time random walks (hereafter CTRW). We recall that a CTRW is defined by a
sequence of i.i.d. jumps $Y_i\in \mathbb{R}^d$ separated by i.i.d.
inter-arrival times $J_i\in \mathbb{R}^+$.

In the special case where $(X_1(t), \dots, X_d(t))$ is a Brownian motion and
$L(t)$ is the inverse of a stable subordinator, then the process (\ref%
{processo subordinato classico}) is a so-called subdiffusion, which has
great interest in many areas of statistical physics (on this point, consult
e.g. \cite{magdziarz}, \cite{magdziarz 2} and \cite{metzler}). See also \cite%
{kumar} and \cite{kumar2}, for the more general case where the external
process is the fractional Brownian motion. Basically, a subdiffusion models
the case where the particle is subject to a trapping effect which delays the
motion with respect to the simple Brownian process; clearly, in this model
it is assumed that the trapping effect is the same in all directions (the
external medium is isotropic) and thus the time-changed process is isotropic
as well as the Brownian motion. It is natural, then, to ask what happens if
the motion of the particle takes place in an anisotropic medium, where the
trapping effect is different depending on the direction. This is the
strongest physical argument that inspired our investigation of processes in (%
\ref{processo subordinato nostro}).

To our knowledge, the study of the model (\ref{processo subordinato nostro})
is a completely new problem. In this paper we restrict our analysis to the
case $d=2$. Our choice is motivated by technical problems; indeed, some of
the results presented in this paper hold only for $d=2$, and, moreover, some
calculations are quite cumbersome even with this restriction. However, we
are sure that the discussion of the 2-dimensional case, besides making the
present exposure clearer, will be useful for possible future studies of the
multidimensional counterpart.

Now we give a brief description of the main results and the outline of the
paper. In Section 2 we recall the definition of bivariate subordinators and
we find some auxiliary results. A crucial part of the paper is Section 3,
where we study in depth some distributional properties and a governing
equation of the biparameter process $\bigl (L_1(t_1), L_2(t_2) \bigr )$; in
this way we present an extension of the well-known theory of inverse
subordinators to the 2-dimensional case. In Section 4 we focus on the
biparameter process $\bigl (X_1(L_1(t_1)), X_2(L_2(t_2)) \bigr )$ and we
prove that its density $p(x_1, x_2, t_1, t_2)$ solves an equation of the
following form
\begin{align}
\mathcal{D}_{t_1, t_2} p(x_1, x_2, t_1, t_2) = \mathcal{G} p(x_1, x_2, t_1,
t_2), \qquad x_1, x_2\neq 0,  \label{eq 2 intro}
\end{align}
where
\begin{align*}
\mathcal{D}_{t_1, t_2} h(t_1, t_2) := \int _0^\infty \int _0^\infty \bigl ( %
h(t_1, t_2) - h(t_1-\tau _1, t_2-\tau _2) \bigr ) \phi (d\tau _1,d\tau _2),
\end{align*}
$\phi$ is the L\'evy measure of the underlying bi-dimensional subordinator $%
(H_1(t_1), H_2(t_2))$ and $\mathcal{G}$ is the dual to the generator of $%
(X_1(t), X_2(t))$. The operator $\mathcal{D}_{t_1, t_2}$ in (\ref{eq 2 intro}%
) is a bi-dimensional version of the generalized fractional derivative
appearing in (\ref{eq 1 intro}); it was already introduced in \cite%
{meerschaert benson} in advection-dispersion equations governing
multidimensional stable L\'evy motions, though acting on the space variables.

Finally, in Section 5, we focus on the special case where $(X_1(t), X_2(t))$
is a Brownian motion $(B_1(t), B_2(t))$, while $\bigl (L_1(t_1)), L_2(t_2) %
\bigr )$ is the inverse of a bivariate stable subordinator. In particular,
we illustrate how to construct a CTRW converging to $\bigl (B_1(L_1(t)),
B_2(L_2(t)) \bigr)$ under a suitable scaling limit.

\vspace{0.1cm}

\textit{Notation:} For simplicity, throughout the paper we will often denote
a process $\{X(t), t\geq 0\}$ by $X(t)$; moreover, we will often denote a
bi-paramenter process $\{(X(t_1), Y(t_2)), t_1\geq 0, t_2\geq 0\}$ by $%
(X(t_1), Y(t_2))$.

\section{Bivariate subordinators}

This section is devoted to bivariate subordinators in the sense of \cite%
{sato}. In Subsection \ref{subsection d}, we review some known facts and we
connect them to the general theory of multivariate L\'evy processes and
their governing equation (see \cite{meer libro}, Chapt. 6, and also \cite%
{garra} and \cite{meerschaert benson} for the special case of multivariate
stable processes). Then, in Subsection \ref{subsection a} we will present
some original results.

\subsection{Preliminaries}

\label{subsection d} A L\'evy process $\{(H_1 (t), H_2(t)), t\geq 0\}$ is
said to be a bivariate subordinator if its components are both a.s.
non-decreasing; we restrict our attention to the case of pure jump
subordinators, having no drift. The characteristic function of $(H_1 (t),
H_2(t))$ can be written (see \cite{sato}) as
\begin{align}
\mathbb{E}e^{i (\alpha H_1(t) +\beta H_2 (t))} = e^{t \sigma (\alpha,
\beta)}, \qquad \alpha, \beta \in \mathbb{R},
\label{funzione caratteristica subordinatore bivariato}
\end{align}
where
\begin{align}
\sigma (\alpha, \beta)= \int \int _{\mathbb{R}_2^+} (e^{i (\alpha x_1+\beta
x_2)}-1) \phi (dx_1,dx_2)
\label{esponente caratteristico bivariate subordinator}
\end{align}
and $\phi (dx_1,dx_2)$ denotes the L\'evy measure; thus $\phi$ is a measure
on $\mathbb{R}_+^2 = \{ (x_1, x_2)\in \mathbb{R}^2: x_1\geq 0, x_2 \geq 0\}$
such that
\begin{align*}
\int \int _{\mathbb{R}_2^+} \min (\sqrt{x_1^2+x_2^2},1) \phi
(dx_1,dx_2)<\infty.
\end{align*}

An equivalent formulation can be given in terms of the Laplace transforms,
i.e.
\begin{align}
\mathbb{E} e^{-\eta _1 H_1(t)-\eta _2H_2(t)} = e^{-tS(\eta _1, \eta _2)}
\qquad \eta _1, \eta _2\geq 0,  \label{abc}
\end{align}
where
\begin{align}
S(\eta _1, \eta _2)= \int \int _{\mathbb{R}_2^+} \bigl (1-e^{-\eta
_1x_1-\eta_2 x_2} \bigr)\phi (dx_1,dx_2)
\label{esponente di Laplace bivariato}
\end{align}
is a bivariate Bernstein function (in the sense of \cite{bochner}, Chapt.
4). The marginal processes are univariate subordinators, hence
\begin{align}
\mathbb{E}e^{-\eta _k H_k (t)}= e^{-tT_k(\eta _k)} \qquad \eta _k\geq 0
\qquad k=1,2,  \label{esponente di Laplace univariato 1}
\end{align}
where
\begin{align}
T_k (\eta _k)= \int _0^\infty (1-e^{-\eta _k x}) \nu _k (dx)
\label{esponente di Laplace univariato 2}
\end{align}
is a Bernstein function and $\nu _k$ is the L\'evy measure associated to $%
H_k $. Throughout this paper we always assume that both $H_1$ and $H_2$ have
infinite activity, namely
\begin{align*}
\nu _k [0,\infty)=\infty \qquad k=1,2.
\end{align*}

It is clear that, if $H_1$ and $H_2$ are independent, then the L\'evy
measure $\phi$ is supported on the coordinate axes, i.e. it has the form
\begin{align}
\phi(dx_1, dx_2)= \nu _1 (dx_1)\delta _0 (dx_2)+ \nu _2(dx_2)\delta _0
(dx_1),  \label{misura di levy nel caso di subordinatori indipendenti}
\end{align}
$\delta _0$ denoting the Dirac delta measure; then
\begin{align}
S(\eta _1, \eta _2)= T_1(\eta _1)+T_2(\eta _2).  \label{S=T1+T2}
\end{align}

\begin{os}
\label{rem:levy-copula} In view of some possible applications, we briefly
recall some connections with the theory of copulas. It is well-known that a
copula allows to separate the dependence structure of a random vector from
its univariate marginal distributions. Moreover, since $\phi $ is a measure,
it is possible to define a suitable notion of copula, that is the L\'{e}vy
copula; see e.g. Section 3 in \cite{KallsenTankov} for its definition and
basic properties (in particular Theorem 3.6. in \cite{KallsenTankov} can be
considered as the analogue of the Sklar's Theorem for L\'{e}vy copulas). L%
\'{e}vy copulas characterize the dependence among components of vector
valued L\'{e}vy processes and they are also used in some estimation problems
(see e.g. \cite{EsmaeiliKluppelberg2011} and \cite{EsmaeiliKluppelberg2013}).
\end{os}

In general, if $(H_1(t), H_2(t))$ has joint density $q_*(x_1, x_2, t)$, then
(see \cite{meer libro}, Chapt. 6]) the following equation holds
\begin{align}
\frac{\partial}{\partial t}q_*(x_1,x_2,t)=- \mathcal{D}_{x_1, x_2}q_*(x_1,
x_2, t),  \label{equazione compatta processo diretto}
\end{align}
where
\begin{align}
\mathcal{D}_{x_1, x_2}h(x_1, x_2):= \int \int _{R _+^2}\biggl (h(x_1, x_2) -
h (x_1-y_1, x_2 - y_2) \biggr ) \phi (dy_1, dy_2)  \label{Marchaud}
\end{align}
on a suitable class of functions $h$. The operator (\ref{Marchaud}) can also
be defined by means of its Laplace symbol:
\begin{align}
\int _0^\infty \int _0^\infty e^{-\eta _1x_1-\eta _2x_2} \mathcal{D}_{x_1,
x_2}f(x_1, x_2) dx_1dx_2=S(\eta _1, \eta _2) \hat{f} (\eta _1, \eta _2)
\label{simbolo di laplace}
\end{align}
where $\hat {f}(\eta _1, \eta _2):= \int _0^\infty \int _0^\infty e^{-\eta
_1x_1-\eta _2x_2}f(x_1, x_2) dx_1dx_2$ and $S$ has been defined in (\ref%
{esponente di Laplace bivariato}). Equation (\ref{equazione compatta
processo diretto}) comes up in the more general setting of Proposition \ref%
{equazione diretto} below.

The above facts extend the best known results holding for univariate
subordinators; indeed, for each $i=1,2$, the marginal density $q_i(x,t)$ of $%
H_i(t)$ solves
\begin{equation*}
\frac{\partial}{\partial t}q_i(x,t)=- \mathcal{D}^{(i)}_{x}q_i(x, t),
\end{equation*}
where the operator on the right-hand side is
\begin{align}
\mathcal{D} ^{(i)}_{x} h(x)= \int _0^\infty \bigl (h(x)-h(x-y)\bigr ) %
\nu_i(dy),  \label{op univ}
\end{align}
whose Laplace symbol is defined by
\begin{equation*}
\int _0^{\infty} e^{-\eta x} \mathcal{D} ^{(i)}_{x} h(x) dx= T_i(\eta )
\tilde{h}(\eta ).
\end{equation*}

Now we present two examples of bivariate subordinators.

\begin{ex}
\label{esempio stabile} Let $(H_1 ^\alpha (t), H_2^\alpha (t))$ be a
bivariate stable subordinator of index $\alpha \in (0,1)$, i.e. a bivariate
stable process with non decreasing components. Its L\'evy measure can be
expressed in polar coordinates as
\begin{align}
\phi_{\alpha} (dr,d\theta)=\frac{C\alpha }{r^{1+\alpha}} drM(d\theta) \qquad
0 \leq \theta \leq \frac{\pi}{2},  \label{misura di levy caso stabile}
\end{align}
where $M(d\theta)$ is a measure on the arc of circle $\{(x,y) \in \mathbb{R}%
_+^2 :x^2+y^2=1\}$. If $H_1 ^\alpha (t)$ and $H_2^\alpha (t)$ are
independent, we can write
\begin{equation*}
M(d\theta)= p\delta _0 (d\theta)+ (1-p) \delta _{\frac{\pi}{2}}(d\theta)
\end{equation*}
for some $p\in (0,1)$, where $\delta _x$ denotes the Dirac delta in $x$.

By simple calculations we see that the characteristic exponent (\ref%
{esponente caratteristico bivariate subordinator}) has the form
\begin{align}
\sigma_{\alpha} (\gamma, \beta)& = \int \int _{\mathbb{R}_2^+} (e^{i (\gamma
x_1+\beta x_2)}-1) \phi_{\alpha} (dx_1,dx_2)  \notag \\
&= \int _0 ^{\pi/2} \biggl ( \int_0^\infty (e^{i r(\gamma \cos \theta+\beta
\sin \theta)}-1) \frac{C\alpha}{r^{\alpha+1}} dr \biggr ) M(d\theta)  \notag
\\
&= -C\Gamma (1-\alpha)\int _0^{\pi/2} [- i(\gamma \cos \theta +\beta \sin
\theta)]^\alpha M(d\theta)  \label{simbolo nel caso stabile}
\end{align}
where, for the last equality, we have taken into account the well-known
formula
\begin{align*}
\int _0^\infty (e^{i\xi x}-1) \frac{\alpha dx}{x^{\alpha +1}}=-\Gamma
(1-\alpha) (-i\xi)^\alpha.
\end{align*}
Hence, the Laplace exponent has the form
\begin{align}
S_{\alpha} (\eta _1, \eta _2)= C\Gamma (1-\alpha)\int _0^{\pi/2} \bigl (\eta
_1 \cos \theta + \eta _2 \sin \theta \bigr )^\alpha M(d\theta).  \label{lap}
\end{align}
Observe that $(\eta _1 \cos \theta + \eta _2 \sin \theta)$ is the Laplace
symbol of $(\cos \theta \frac{\partial}{\partial x_1}+ \sin \theta \frac{%
\partial}{\partial x_2})$, which is the directional derivative along the
unit vector $(\cos \theta, \sin \theta )$. Then, by taking the Laplace
inverse of (\ref{lap}), the operator in (\ref{Marchaud}) here reduces to the
following pseudo-differential operator
\begin{align}
\mathcal{D}^\alpha _{x_1,x_2} f(x_1, x_2)= C\Gamma (1-\alpha)\int _0^{\pi/2} %
\bigl ( \cos \theta \frac{\partial}{\partial x_1} + \sin \theta \frac{%
\partial}{\partial x_2} \bigr )^\alpha M(d\theta).  \label{pseudo}
\end{align}
Note that (\ref{pseudo}) represents the average under $M(d\theta)$ of the
fractional power of the directional derivative along $(\cos \theta, \sin
\theta)$. Thus equation (\ref{equazione compatta processo diretto}) takes
the form
\begin{equation*}
\frac{\partial}{\partial t} q_* (x_1, x_2, t)= -\mathcal{D}^\alpha_{x_1,
x_2}q_* (x_1, x_2,t)
\end{equation*}
which is a particular case of the equation governing multivariate stable
processes (on this point consult \cite{garra} and \cite{meerschaert benson}).
\end{ex}

\begin{ex}
Let $Y_1(t)$, $Y_2(t)$ and $Z(t)$ be independent subordinators, and let $%
T_1(\cdot)$, $T_2(\cdot)$ and $G(\cdot)$ be the respective Bernstein
functions. For any $c_1, c _2 \in \mathbb{R}^+$, we consider the bivariate
subordinator
\begin{equation}  \label{semeraro}
\begin{cases}
H_1(t)= Y_1(t)+c _1 Z(t) \\
H_2(t)= Y_2(t)+c _2 Z(t).%
\end{cases}%
\end{equation}
Such example is motivated by the model studied in \cite{Semeraro} and \cite%
{LucianoSemeraro}; here the authors considered multivariate subordinators
such that each component is a sum of an idiosyncratic and a common term,
which has some interest in finance modelling. It is easy to see that (\ref%
{semeraro}) is characterized by the bivariate Bernstein function
\begin{align*}
S(\eta _1, \eta _2)= T_1 (\eta _1)+ T_2(\eta _2) + G(c _1\eta _1+ c _2\eta
_2).
\end{align*}
Then the support of the L\'evy measure is the union of the coordinate axes
and the line with direction $(c _1, c_2)$ passing through the origin; for
the explicit expression of the L\'evy measure consult Prop. 3.1 in \cite%
{Semeraro}.

It is interesting to note that, if $Y_1(t)$, $Y_2(t)$ and $Z(t)$ are stable
subordinators with parameter $\alpha \in (0,1)$, then the bivariate
Bernstein function reads
\begin{equation*}
S(\eta _1, \eta _2)= \eta _1^\alpha + \eta _2 ^\alpha + (c _1\eta _1+c
_2\eta _2)^\alpha;
\end{equation*}
thus the operator (\ref{Marchaud}) reduces to the pseudo-differential
operator
\begin{align*}
\mathcal{D}_{x_1, x_2} = \frac{\partial ^\alpha}{\partial x_1^\alpha} +
\frac{\partial ^\alpha}{\partial x_2^\alpha} + \biggl (c _1 \frac{\partial}{%
\partial x_1}+c _2 \frac{\partial }{\partial x_2} \biggr)^{\alpha },
\end{align*}
where the last summand is the fractional power of the directional derivative.
\end{ex}

\subsection{Some results}

\label{subsection a}

For technical reasons that will be clear in the next section, it is
sometimes convenient to consider the related bi-parameter process
\begin{equation*}
\{ \bigl (H_1 (t_1), H_2(t_2) \bigr), t_1 \geq 0, t_2 \geq 0\}
\end{equation*}
rather than the process $\{(H_1(t), H_2(t)), t\geq 0\}$. The following
result characterizes its distribution.

\begin{prop}
\label{esponente caratteristico campo aleatorio} The double Laplace
transform of $\bigl (H_1 (t_1), H_2(t_2) \bigr)$ reads
\begin{align}
\mathbb{E} e^{- (\alpha H_1(t_1)+ \beta H_2(t_2))}& = e^{-t _1 S (\alpha,
\beta) } e^{-(t_2-t_1) T _2 (\beta)} \mathbf{1} _{\{t_2\geq t_1\}}  \notag \\
&+ e^{-t _2 S (\alpha, \beta) } e^{-(t_1-t_2) T _1 (\alpha)} \mathbf{1}%
_{\{t_1> t_2\}} \qquad \alpha, \beta \geq 0,
\end{align}
where $S$, $T_1$ and $T_2$ have been defined in (\ref{esponente di Laplace
bivariato}) and (\ref{esponente di Laplace univariato 1}), while $\mathbf{1}%
_{A}$ denotes the indicator function of the set $A$.
\end{prop}

\begin{proof}
The desired equality can be checked as follows:
\begin{align*}
&\mathbb{E} e^{- (\alpha H_1(t_1)+ \beta H_2(t_2))} \\
&= \mathbb{E} \bigl [ e^{ - (\alpha H_1 (t_1)+ \beta H_2 (t_1)) }e^{-\beta (
H_2 (t_2)-H_2(t_1))} \bigr ]\mathbf{1} _{\{t_2\geq t_1\}} \\
& +\mathbb{E} \bigl [ e^{ - (\alpha H_1 (t_2)+ \beta H_2 (t_2)) }e^{-\alpha(
H_1 (t_1)-H_1(t_2))} \bigr ] \mathbf{1}_{\{t_1> t_2\}} \\
& = \mathbb{E}[ e^{ -(\alpha H_1 (t_1)+ \beta H_2 (t_1)) }] \, \mathbb{E}
[e^{- \beta ( H_2 (t_2)-H_2(t_1))}] \mathbf{1} _{\{t_2\geq t_1\}} \\
&+ \mathbb{E} [e^{ -(\alpha H_1 (t_2)+ \beta H_2 (t_2)) }] \, \mathbb{E}
[e^{- \alpha ( H_1 (t_1)-H_1(t_2))}] \mathbf{1}_{\{t_1> t_2\}} \\
& = e^{-t _1 S (\alpha, \beta) } e^{-(t_2-t_1) T _2 (\beta)} \mathbf{1}
_{\{t_2\geq t_1\}}+ e^{-t _2 S (\alpha, \beta) } e^{-(t_1-t_2) T _1(\alpha)}
\mathbf{1}_{\{t_1> t_2\}},
\end{align*}
where we repeatedly used independence of the increments.
\end{proof}

From now on we assume that the random vector $(H_1(t), H_2(t))$ has joint
and marginal densities, respectively defined by
\begin{align*}
&P(H_1(t)\in dx_1, H_2(t)\in dx_2):= q_*(x_1, x_2, t)dx_1dx_2, \qquad t>0 \\
&P(H_i(t)\in dx_i):= q_i(x_i, t)dx_i \qquad i=1,2 \qquad t>0.
\end{align*}
Moreover
\begin{align}
P(H_1(0)\in dx_1, H_2(0)\in dx_2):= \delta _0(dx_1) \delta _0(dx_2).
\label{in1}
\end{align}
Hence we introduce the joint density of $(H_1(t_1), H_2(t_2))$ by
\begin{align*}
P\bigl (H_1(t_1)\in dx_1, H_2(t_2)\in dx_2\bigr):= q(x_1, x_2,t_1,
t_2)dx_1dx_2 \qquad \,\, \text{for all} \,\, t_1, t_2> 0,
\end{align*}
where
\begin{align*}
q(x_1, x_2, t_1, t_2)= & \int _0^\infty q_*(x_1,z,t_1)q_2(x_2-z, t_2-t_1)dz
\, \mathbf{1} _{\{t_1\leq t_2\}} \\
+ & \int _0^\infty q_*(z,x_2,t_2)q_1(x_1-z, t_1-t_2)dz \, \mathbf{1}
_{\{t_1> t_2\}},
\end{align*}
which satisfies
\begin{align*}
q(x_1, x_2, t,t)= q_*(x_1, x_2, t).
\end{align*}
Moreover we have that
\begin{align}
\lim _{t_1 \to 0} q(x_1, x_2, t_1,t_2)= \delta _0 (x_1) q_2(x_2, t_2)
\label{in2} \\
\lim _{t_2 \to 0} q(x_1, x_2, t_1,t_2)= \delta _0(x_2)q_1 (x_1, t_1)
\label{in3}
\end{align}
in the sense of distributions.

The following proposition provides a governing equation for the bi-parameter
process $(H_1(t_1), H_2(t_2))$, which involves the operator $\mathcal{D}
_{x_1, x_2 }$ defined in (\ref{Marchaud}). This extends (\ref{equazione
compatta processo diretto}) to the case where $t_1\neq t_2$.

\begin{prop}
\label{equazione diretto} If $t_1\neq t_2$, then joint density $q(x_1, x_2,
t_1, t_2)$ is the fundamental solution of the equation
\begin{align}
\biggl (\frac{\partial}{\partial t_1 } +\frac{\partial}{\partial t_2} %
\biggr) q(x_1, x_2, t_1,t_2) = -\mathcal{D} _{x_1, x_2 } q(x_1, x_2,
t_1,t_2) \qquad \, t_1\neq t_2 \qquad t_1, t_2>0,  \label{equazione q}
\end{align}
under the initial conditions (\ref{in2}) and (\ref{in3}). In the case $%
t_1=t_2=t$, we have
\begin{align}
\frac{\partial}{\partial t} q_*(x_1, x_2, t)= -\mathcal{D} _{x_1, x_2
}q_*(x_1, x_2, t) \qquad t>0,  \label{equazione Q}
\end{align}
under the initial condition (\ref{in1}).
\end{prop}

\begin{proof}
The Laplace transform given by Proposition \ref{esponente caratteristico
campo aleatorio} is differentiable for $t_1\neq t_2$, and, by applying the
operator $\bigl (\frac{\partial}{\partial t_1}+ \frac{\partial }{\partial t_2%
} \bigr)$ to both sides, we have
\begin{align}
\biggl (\frac{\partial }{\partial t_1}+ \frac{\partial}{\partial t_2} %
\biggr) \mathbb{E} e^{-(\alpha H_1(t_1)+ \beta H_2(t_2))}= -S(\alpha ,
\beta) \mathbb{E} e^{- (\alpha H_1(t_1)+ \beta H_2(t_2))}, \qquad t_1\neq
t_2.
\end{align}
Then, taking into account (\ref{simbolo di laplace}), by the Laplace
inversion of both sides we get (\ref{equazione q}). Finally, for $t_1=t_2=t$%
, by (\ref{abc}) we have
\begin{align*}
\frac{\partial}{\partial t} \mathbb{E} e^{-\eta _1 H_1(t)-\eta _2H_2(t)} =
-S(\eta _1, \eta _2)e^{-tS(\eta _1, \eta _2)}
\end{align*}
and the inverse Laplace transform of both sides gives (\ref{equazione Q}).
\end{proof}

We now present a further result that will be useful later. For $t_1,t_2>0$,
we define the tail of the L\'evy measure $\phi$ as
\begin{align}
\bar{\phi}(t_1, t_2)=\phi ((t_1, \infty)\times (t_2, \infty))= \int
_{t_1}^\infty \int _{t_2}^\infty \phi (dx_1,dx_2).
\label{coda misura di levy 2d}
\end{align}
Note that, in the case of independence, formula (\ref{misura di levy nel
caso di subordinatori indipendenti}) yields $\bar{\phi}(t_1, t_2)=0$ for
each $t_1$ and $t_2$, because $\phi$ is supported on the coordinate axes.

\begin{prop}
The double Laplace transform of (\ref{coda misura di levy 2d}) reads
\begin{align}
\int _0^\infty \int _0^\infty e^{ -\eta _1t_1-\eta _2t_2} \bar{\phi} (t_1,
t_2) dt_1dt_2= \frac{T_1(\eta _1)+ T_2(\eta _2)-S(\eta _1, \eta _2)}{\eta _1
\eta _2},  \label{trasformata coda misura di levy}
\end{align}
where $T_1$, $T_2$ and $S$ have been defined in (\ref{esponente di Laplace
univariato 2}) and (\ref{esponente di Laplace bivariato}).
\end{prop}

\begin{proof}
It is sufficient to start from expression (\ref{esponente di Laplace
bivariato}) and apply consecutively the integration by parts in both
variables. The calculations exploit that
\begin{align*}
\nu _1(dx_1) =\int _0 ^\infty \phi (dx_1,dx_2) \qquad \text{and} \qquad \nu
_2(dx_2) =\int _0 ^\infty \phi (dx_1,dx_2).
\end{align*}
The result holds by simple algebraic manipulations.
\end{proof}

\section{Inverses of bivariate subordinators}

Let $\{(H_1(x), H_2(x)), x\geq 0\}$ be a bivariate subordinator. We here
define its inverse $\{(L_1(t), L_2(t)), t\geq 0\}$ as the process with
marginal components
\begin{align*}
L_i (t)= \inf \{ x: H_i(x)>t\}, \qquad i=1,2.
\end{align*}
We also consider the related bi-parameter process $\{(L_1(t_1), L_2(t_2)),
t_1\geq 0, t_2\geq 0\}$, such that
\begin{align*}
L_i (t_i)=\inf \{ x: H_i(x)>t_i\}, \qquad i=1,2;
\end{align*}
then
\begin{align}
P(L_1(t_1)>x_1, L_2(t_2)>x_2)=P(H_1(x_1)\leq t_1, H_2(x_2)\leq t_2).
\label{diretto e inverso}
\end{align}

It is well-known from the theory of univariate subordinators and their
inverses (see, for example, \cite{kolokoltsov}) that the marginal densities $%
l_1(x_1, t_1)$ and $l_2(x_2,t_2)$ solve the generalized fractional Cauchy
problems
\begin{align}  \label{equazioni 1d}
\begin{cases}
\frac{\partial}{\partial x_i} l_i(x_i, t_i)= -\mathcal{D}_{t_i} l_i(x_i, t_i)
\\
l_i(0, t_i)= -\mathcal{D}_{t_i} \theta (t_i) = \overline{\nu}(t_i) \qquad
i=1,2,%
\end{cases}%
\end{align}
where $\theta (\cdot)$ denotes the Heaviside function, $\mathcal{D}_t$ is
the operator defined in (\ref{operatore univariato}), while
\begin{equation*}
\overline{\nu}_k(t_k)= \nu _k [t_k,\infty) \qquad k=1,2
\end{equation*}
is the tail of the L\'evy measure. Moreover, it is well-known that
\begin{align*}
\int _0^{\infty} e^{-\eta _i t_i} l_i(x_i, t_i)dt_i = \frac{T_i(\eta _i)}{%
\eta _i} e^{-T_i(\eta _i)x_i} \qquad i=1,2
\end{align*}
and the space-time Laplace transforms read
\begin{align}
\int _0^{\infty} \int _0^{\infty} e^{-\eta _i t_i} e^{-\xi _i x_i} l(x_i,
t_i)dx_i dt_i = \frac{T_i (\eta _i)}{\eta _i (\xi _i+T_i(\eta _i)) } \qquad
i=1,2.  \label{laplace inverso 1}
\end{align}
The aim of the next subsections is to find the bi-dimensional counterparts
of such results.

\subsection{Distributional properties}

We here assume that $(H_1(x_1), H_2(x_2))$ is an absolutely continuous
random vector with density $q(t_1, t_2, x_1, x_2)$.

In Proposition \ref{trasformate solo rispetto al tempo}, we show that the
distribution of $\bigl (L_1(t_1), L_2(t_2) \bigr)$ has two components. The
first one is absolutely continuous with respect to the bi-dimensional
Lebesgue measure, namely
\begin{align}
P(L_1(t_1)\in dx_1, L_2(t_2)\in dx_2)=l (x_1,x_2,t_1,t_2)dx_1dx_2 \qquad
x_1\neq x_2;  \label{LLL}
\end{align}
the second one has support on the bisector line $x_1=x_2$ with
one-dimensional Lebesgue density, i.e.
\begin{align}
P(L_1(t_1)\in dx, L_2(t_2)\in dx)=l_*(x,t_1,t_2)dx.  \label{LLLL}
\end{align}

The existence of the second component is due to the intuitive fact that the
event $\{L_1(t_1)=L_2(t_2)\}$ has positive probability; in fact, since $H_1$
and $H_2$ have simultaneous jumps, the first time in which $H_1$ goes beyond
the level $t_1$ may coincide with the first time in which $H_2$ goes beyond
the level $t_2$.

In view of what follows, we need to introduce the Laplace transforms of $l$
and $l_*$ with respect to the time-variables, i.e.
\begin{align*}
\tilde{l}(x_1, x_2, \eta _1, \eta _2)&:= \int _0^\infty \int _0^\infty
e^{-\eta _1t_1-\eta _2t_2} l(x_1, x_2, t_1, t_2)dt_1dt_2, \\
\tilde{l}_*(x,\eta_1, \eta_2)&:= \int _0^\infty \int _0^\infty e^{-\eta
_1t_1-\eta _2t_2}l_*(x, t_1, t_2)dt_1dt_2.
\end{align*}

\begin{prop}
\label{trasformate solo rispetto al tempo} Assume that, for each $x_1,x_2>0$%
, $(H_1(x_1), H_2(x_2))$ has a density denoted by $q(t_1, t_2, x_1, x_2)$.
Then, for any $t_1, t_2>0$, the distribution of $(L_1(t_1), L_2(t_2))$ has
an absolutely continuous component with density
\begin{align}
l(x_1, x_2, t_1, t_2)= \frac{\partial ^2}{\partial x_1 \partial x_2} \int
_0^{t_1} \int _0^{t_2} q(\tau_1, \tau _2, x_1, x_2) d\tau _1 d\tau _2,
\qquad x_1\neq x_2,  \label{densita inverso}
\end{align}
and its Laplace transform is
\begin{align}
&\tilde{l}(x_1, x_2, \eta _1, \eta _2)  \label{uvz} \\
&= \frac{1}{\eta _1 \eta _2}T_2(\eta _2) (S(\eta _1, \eta _2)-T_2(\eta
_2))e^{-x_1(S(\eta _1, \eta _2)-T_2(\eta _2))}e^{-x_2T_2(\eta _2)} \mathbf{1}%
_{\{x_1<x_2\}}  \notag \\
&+\frac{1}{\eta _1 \eta _2} T_1(\eta _1)(S(\eta _1, \eta _2)-T_1(\eta
_1))e^{-x_2(S(\eta _1, \eta _2)-T_1(\eta _1))}e^{-x_1T_1(\eta _1)} \mathbf{1}%
_{\{x_1>x_2\}},  \notag
\end{align}
where $\eta _1, \eta _2\geq 0$ and $S$, $T_1$, $T_2$ have been defined in (%
\ref{esponente di Laplace bivariato}) and (\ref{esponente di Laplace
univariato 2}). Moreover, the distribution of $(L_1(t_1), L_2(t_2))$ also
has a singular component supported on the line $x_1=x_2$, having density
\begin{align}
l_*(x,t_1, t_2)= \frac{\partial^2}{\partial x^2} \int _0^{t_1}\int _0^{t_2}
q_*(\tau_1, \tau _2, x)d\tau _1d\tau _2,  \label{densita 1d inverso}
\end{align}
whose Laplace transform reads
\begin{align}
\tilde{l}_*(x,\eta_1, \eta_2) = \frac{T_1(\eta _1) +T_2 (\eta _2)-S(\eta _1,
\eta _2))}{\eta _1\eta _2} e^{-xS(\eta _1, \eta _2)}.
\label{seconda trasformata}
\end{align}
\end{prop}

\begin{proof}
Since $(H_1(x_1), H_2(x_2))$ has a density, it is possible to differentiate
both sides of (\ref{diretto e inverso}) and we get (\ref{densita inverso}).
For $x_1<x_2$, by using Proposition \ref{esponente caratteristico campo
aleatorio} and equation (\ref{densita inverso}), we have
\begin{align*}
& \int _0^\infty \int _0^\infty e^{-\eta _1t_1-\eta _2t_2} l(x_1, x_2, t_1,
t_2) dt_1dt_2 \\
& = \int _0^\infty \int _0^\infty e^{-\eta _1t_1-\eta _2t_2} \frac{\partial
^2}{\partial x_1 \partial x_2} \int _0^{t_1} \int _0^{t_2} q(\tau_1, \tau
_2, x_1, x_2) d\tau _1 d\tau _2 \, dt_1dt_2 \\
& = \frac{\partial ^2}{\partial x_1 \partial x_2} \int _0^\infty \int
_0^\infty q(\tau _1, \tau _2, x_1, x_2) \int _{\tau _1}^\infty \int _{\tau
_2}^\infty e^{-\eta _1t_1-\eta _2t_2}dt_1dt_2 \, d\tau _1d\tau _2 \\
& = \frac{1}{\eta _1\eta _2} \frac{\partial ^2}{\partial x_1 \partial x_2}
\int _0^\infty \int _0^\infty e^{-\eta _1\tau _1-\eta _2\tau _2} q(\tau _1,
\tau _2, x_1, x_2) d\tau _1d\tau _2 \\
& = \frac{1}{\eta _1\eta _2} \frac{\partial ^2}{\partial x_1 \partial x_2}
e^{-x_1S(\eta _1, \eta _2)- (x_2-x_1)T_2(\eta _2)} \\
&=\frac{1}{\eta _1 \eta _2} T_2(\eta _2) (S(\eta _1, \eta _2)-T_2(\eta
_2))e^{-x_1(S(\eta _1, \eta _2)-T_2(\eta _2))}e^{-x_2T_2(\eta _2)}.
\end{align*}
The result for the case $x_1>x_2$ can be proved similarly. We also remark
that
\begin{align*}
\int _0^\infty \int _0^\infty l(x_1, x_2, t_1, t_2)dx_1 dx_2\neq 1,
\end{align*}
as
\begin{align*}
&\int _0^\infty \int _0 ^\infty e^{-\eta _1t_1-\eta _2t_2} \biggl (\int
_0^\infty \int _0^\infty l(x_1, x_2, t_1, t_2)dx_1dx_2 \biggr) dt_1dt_2 \neq
\frac{1}{\eta _1\eta _2}.
\end{align*}
This proves that $P\bigl (L_1(t_1)=L_2(t_2)\bigr )>0$. By putting $x_1=x_2=x$
in (\ref{diretto e inverso}), and taking the derivatives with respect to $x$
of both sides, equation (\ref{densita 1d inverso}) is obtained. To compute (%
\ref{seconda trasformata}) we can use the following representation
\begin{align*}
l_*(x,t_1, t_2)= \int _0^{t_1} \int _0^{t_2} \bar{\phi} (t_1-s_1, t_2-s_2)
q_*(s_1, s_2,x)ds_1ds_2
\end{align*}
and apply the convolution theorem for the Laplace transforms, by considering
expressions (\ref{trasformata coda misura di levy}) and (\ref{abc}).
\end{proof}

\begin{os}
Of course, the following normalizing condition holds:
\begin{align}
\int _0^\infty \int _0^\infty l(x_1, x_2, t_1, t_2)dx_1dx_2+ \int _0^\infty
l_*(x,t_1, t_2)dx=1 .  \label{condizione di normalizzazione}
\end{align}
In fact we can check (\ref{condizione di normalizzazione}) by using the
Laplace transforms in Proposition \ref{trasformate solo rispetto al tempo},
and we get
\begin{align*}
&\int _0^\infty \int _0 ^\infty e^{-\eta _1t_1-\eta _2t_2} \biggl (\int
_0^\infty \int _0^\infty l(x_1, x_2, t_1, t_2)dx_1dx_2+ \int _0^\infty
l_*(x,t_1, t_2)dx \biggr) dt_1dt_2 \\
&=\frac{1}{\eta _1\eta _2}.
\end{align*}
\end{os}

\begin{os}
Proposition \ref{trasformate solo rispetto al tempo} gives the time-Laplace
transform of the distribution, presenting three different expressions
related to the regions $x_1<x_2$, $x_1>x_2$ and $x_1=x_2$. These three
components are jointly taken into account in the following space Laplace
transform:
\begin{align}
& \int _0^{\infty} \int _0^\infty \int_{0}^{\infty }\int_{0}^{\infty
}e^{-\eta _{1}t_{1}-\eta _{2}t_{2}} e^{-\xi _1x_1 -\xi _2 x_2} P(L_1(t_1)\in
dx_1, L_2(t_2)\in dx_2) dt_1dt_2  \notag \\
& =\int _0^{\infty} \int _0^\infty e^{-\xi _1x_1 -\xi _2 x_2} \tilde{l}(x_1,
x_2, \eta_1,\eta_2) dx_1dx_2 + \int _0^{\infty} e^{-(\xi _1 +\xi _2)x }
\tilde{l}_*(x,\eta_1,\eta_2) dx  \notag \\
&=\frac{T_1(\eta _{1})T_2(\eta _{2})}{\eta _{1}\eta _{2}\left[ \xi
_{1}+T_1(\eta _{1})\right] \left[ \xi _{2}+T_2(\eta _{2})\right] }+\frac{\xi
_{1}\xi _{2}\left[ T_1(\eta _{1})+T_2(\eta _{2})-S(\eta _{1},\eta _{1})%
\right] }{\eta _{1}\eta _{2}\left[ \xi _{1}+T_1(\eta _{1})\right] \left[ \xi
_{2}+T_2(\eta _{2})\right] \left[ \xi _{1}+\xi _{2}+S(\eta _{1},\eta _{1})%
\right] }  \label{laplace inverso}
\end{align}
for $\xi _{1},\xi _{2},\eta _{1},\eta _{2}>0$. We remark that, in the case
of independence, by using (\ref{S=T1+T2}) we have that (\ref{laplace inverso}%
) reduces to the product of the univariate transforms given in (\ref{laplace
inverso 1}).
\end{os}

In the following proposition we prove that the survival function of $%
(L_1(t_1), L_2(t_2))$ decays at least exponentially fast. Therefore all the
mixed moments of integer order exist. We then compute the Laplace transform
for the covariance between $L_1(t_1)$ and $L_2(t_2)$.

\begin{prop}
For any $\eta_1, \eta_2\geq 0$ the following inequality holds
\begin{align}
P\bigl (L_1(t_1)\geq x_1, L_2(t_2)\geq x_2\bigr )\leq & \, e^{-x _1 S (\eta
_1, \eta _2) } e^{-(x_2-x_1) T _2 (\eta _2)} e^{\eta _1t_1} e^{\eta _2t_2}
\mathbf{1} _{\{x_2\geq x_1\}}  \notag \\
&+ e^{-x _2 S (\eta _1, \eta _2) } e^{-(x_1-x_2) T _1 (\eta _1)} e^{\eta
_1t_1} e^{\eta _2t_2} \mathbf{1}_{\{x_1> x_2\}}.  \label{survival function}
\end{align}
Therefore, all the mixed moments of integer order exist and, moreover,
\begin{align}
\int _0^\infty \int _0^\infty e^{-\eta _1t_1-\eta _2t_2}cov \bigl ( %
L_1(t_1),L_2(t_2) \bigr )dt_1dt_2= \frac{T_1(\eta _1)+T_2(\eta _2)-S(\eta
_1, \eta _2)}{\eta _1\eta _2T_1(\eta _1)T_2(\eta _2)S(\eta _1, \eta _2)}.
\label{covarianza}
\end{align}
\end{prop}

\begin{proof}
The proof of (\ref{survival function}) is based on the so-called bivariate
Markov inequality (see \cite{markov inequality}, Theorem 1). For $\eta _1,
\eta _2\geq 0$,
\begin{align*}
P\bigl (L_1(t_1)\geq x_1, L_2(t_2)\geq x_2\bigr )&= P\bigl (H_1(x_1)<t_1,
H_2(x_2)< t_2\bigr ) \\
& = P\bigl ( e^{-\eta _1H_1(x_1)}>e^{-\eta _1t_1}, e^{-\eta _2 H_2(x_2)}
>e^{-\eta _2t_2} \bigr ) \\
& \leq \frac{\mathbb{E} [e^{-\eta _1H_1(x_1) -\eta _2H_2(x_2) }] }{e^{-\eta
_1t_1} e^{-\eta _2t_2} } \\
& = e^{-x _1 S (\eta _1, \eta _2) } e^{-(x_2-x_1) T _2 (\eta _2)} e^{\eta
_1t_1} e^{\eta _2t_2} \mathbf{1} _{\{x_2\geq x_1\}} \\
&+ e^{-x _2 S (\eta _1, \eta _2) } e^{-(x_1-x_2) T _1 (\eta _1)} e^{\eta
_1t_1} e^{\eta _2t_2} \mathbf{1}_{\{x_1> x_2\}}
\end{align*}
where the last equality holds by Proposition \ref{esponente caratteristico
campo aleatorio}. The Laplace transform of $\mathbb{E}\bigl (L_1(t_1)L_2(t_2)%
\bigr)$ can be computed as
\begin{align*}
& \int _0^\infty \int _0^\infty e^{-\eta _1t_1-\eta _2t_2}\mathbb{E}\bigl ( %
L_1(t_1)L_2(t_2) \bigr )dt_1dt_2 \\
&= \int _0^\infty \int _0^\infty x_1x_2 \tilde{l} (x_1, x_2, \eta _1, \eta
_2)dx_1dx_2 + \int _0^\infty x^2 \, \tilde{l}_*(x, \eta _1, \eta _2)dx
\end{align*}
and then, exploiting Prop. \ref{trasformate solo rispetto al tempo}, after
some calculations we get
\begin{align}
\int _0^\infty \int _0^\infty e^{-\eta _1t_1-\eta _2t_2}\mathbb{E}\bigl ( %
L_1(t_1)L_2(t_2) \bigr )dt_1dt_2= \frac{T_1(\eta _1)+T_2(\eta _2)}{\eta
_1\eta _2T_1(\eta _1)T_2(\eta _2)S(\eta _1, \eta _2)}.  \label{momento misto}
\end{align}
Finally, by taking into account that univariate inverse subordinators are
such that
\begin{align*}
\int _0^\infty e^{-\eta_i t_i} \mathbb{E} L_i(t_i)dt_i= \frac{1}{\eta
_iT_i(\eta _i)} \qquad i=1,2,
\end{align*}
we immediately obtain (\ref{covarianza}).
\end{proof}

\begin{os}
It is known that, if $H$ is a univariate subordinator and $L$ is its
inverse, then $\mathbb{E}L(t)$ can be related to the potential density of $H$%
. We recall that, in the frame of potential theory (see e.g. \cite{bogdan}),
the potential density of a subordinator $H$ is the function $t\to u(t)$ such
that $u(t)dt$ represents the mean sojourn time spent by $H$ in the interval $%
[t, t+dt)$, namely
\begin{align*}
u(t)dt= \int _0^\infty P(H(x)\in dt)dx;
\end{align*}
indeed the following relation holds
\begin{align*}
u(t)=\frac{d}{dt} \mathbb{E}L(t).
\end{align*}
It is interesting to remark that a similar result holds for a bivariate
subordinator $(H_1, H_2)$ with inverse $(L_1, L_2)$. Indeed, let $u(t_1,
t_2)dt_1dt_2$ be the mean sojourn time spent by $(H_1, H_2)$ in the set $%
[t_1, t_1+dt_1)\times [t_2, t_2+dt_2)$, i.e.
\begin{align*}
u(t_1, t_2)dt_1dt_2= \int \int _{\mathbb{R}_+^2} P(H_1(x_1)\in dt_1,
H_2(x_2)\in dt_2) dx_1dx_2;
\end{align*}
then the following relation holds
\begin{align*}
\frac{\partial ^2}{\partial t_1\partial t_2} \mathbb{E}\bigl ( %
L_1(t_1)L_2(t_2) \bigr) &= \frac{\partial ^2}{\partial t_1\partial t_2} \int
\int _{\mathbb{R}_+^2} P(L_1(t_1)>x_1, L_2(t_2)>x_2)dx_1dx_2 \\
&= \frac{\partial ^2}{\partial t_1\partial t_2}\int \int _{\mathbb{R}_+^2}
P(H_1(x_1)\leq t_1, H_2(x_2) \leq t_2)dx_1dx_2 \\
&=u(t_1, t_2).
\end{align*}
\end{os}

\subsection{Governing equation}

The following theorem provides boundary value problems for both the
densities $l(x_1, x_2, t_1, t_2)$ and $l_*(x,t_1, t_2)$ defined in (\ref{LLL}%
) and (\ref{LLLL}). In this way we generalize the boundary value problems in
(\ref{equazioni 1d}).
%Moreover, the functions $l_1$ and $l_2$, which are solutions of (\ref{equazioni 1d}) are here considered known functions and thus they are used to set the new Cauchy problems.
The arising equations are PDEs exhibiting ordinary derivatives in the space
variables and the integro-differential operator $\mathcal{D}_{t_1, t_2}$ in
the time variables. We note that, for well-posedness reasons, we set two
distinct boundary value problems for $l$, which are respectively defined on
the open sets $x_2>x_1>0$ and $x_1>x_2>0$. We will denote again the
Heaviside function by $\theta (\cdot)$.

\begin{te}
\label{teorema equazione governante} The density $l(x_1, x_2, t_1, t_2)$
solves the equation
\begin{align}
&\biggl (\frac{\partial}{\partial x_1 } +\frac{\partial}{\partial x_2} %
\biggr) l(x_1, x_2, t_1,t_2) =- \mathcal{D} _{t_1, t_2 } l(x_1, x_2, t_1,t_2)
\label{equazione pura}
\end{align}
on the open set $x_2>x_1>0$, under the boundary condition
\begin{align}
&l(0,x_2,t_1,t_2) = (\mathcal{D} _{t_1, t_2 } - \mathcal{D}^{(2)}_{t_2}) \,
\mathcal{D} ^{(2)}_{t_2} \theta (t_1) P(L_2(t_2)\geq x_2), \qquad x_2>0,
\label{seconda condizione iniziale}
\end{align}
where $\mathcal{D} ^{(i)}_{x}$ has been defined in $(\ref{op univ})$.
Moreover, on the open set $x_1>x_2>0$, $l(x_1, x_2, t_1, t_2)$ solves the
same equation (\ref{equazione pura}) under the boundary condition
\begin{align}
&l(x_1,0,t_1,t_2)= (\mathcal{D} _{t_1, t_2 } - \mathcal{D}^{(1)}_{t_1}) \,
\mathcal{D}^{(1)}_{t_1} \theta (t_2) P(L_1(t_1)\geq x_1) \qquad x_1>0.
\label{prima condizione iniziale}
\end{align}
Furthermore, the density $l_*(x,t_1, t_2)$ solves the equation
\begin{align}
\frac{\partial}{\partial x}l_*(x,t_1, t_2)= -\mathcal{D} _{t_1, t_2
}l_*(x,t_1, t_2) \qquad x>0  \label{equazione densita 1d}
\end{align}
under the boundary condition
\begin{align}
l_*(0,t_1, t_2)= \bar{\phi} (t_1, t_2) ,  \label{condizione iniziale 1d}
\end{align}
where $\bar{\phi}$ has been defined in (\ref{coda misura di levy 2d}).
\end{te}

\begin{proof}
We start with the case $x_2>x_1>0$; the case $x_1>x_2>0$ can be treated
similarly. We apply $\frac{\partial}{\partial x_1}+\frac{\partial}{\partial
x_2}$ to both members of (\ref{densita inverso}); by using Proposition \ref%
{equazione diretto} we have
\begin{align}
\biggl ( \frac{\partial}{\partial x_1}+\frac{\partial}{\partial x_2} \biggr
) l(x_1, x_2, t_1, t_2) = -\frac{\partial ^2}{\partial x_1\partial x_2} \int
_0^{t_1} \int _0 ^{t_2} \mathcal{D}_{\tau _1, \tau _2} q(\tau_1, \tau_2,
x_1, x_2) d\tau _1 d\tau _2.
\end{align}
By using the definition of $\mathcal{D}_{\tau _1, \tau _2}$, Fubini's
theorem gives
\begin{align}
\int _0^{t_1} \int _0 ^{t_2} \mathcal{D}_{\tau _1, \tau _2} q(\tau_1,
\tau_2, x_1, x_2) d\tau _1d\tau _2= \mathcal{D}_{t _1, t _2} \int _0^{t_1}
\int _0 ^{t_2} q(\tau_1, \tau_2, x_1, x_2) d\tau _1d\tau _2
\label{commutes with the integral}
\end{align}
which leads to equation (\ref{equazione pura}). In order to check the
boundary condition (\ref{seconda condizione iniziale}), we put $x_1=0$ in (%
\ref{uvz}), which yields
\begin{align*}
\tilde{l}(0, x_2, \eta _1, \eta _2) =\frac{1}{\eta _1\eta _2}T_2(\eta _2)
(S(\eta _1, \eta _2)-T_2(\eta _2))e^{-x_2T_2(\eta _2)}
\end{align*}
and the inversion of the Laplace transform gives (\ref{seconda condizione
iniziale}).

We finally conclude with the results for the density $l_*(x,t_1, t_2)$. In
order to obtain equation (\ref{equazione densita 1d}) we have to apply $%
\frac{\partial}{\partial x}$ to both sides of (\ref{densita 1d inverso}) and
use (\ref{equazione compatta processo diretto}) together with the fact that $%
\mathcal{D}_{t_1, t_2}$ commutes with the integral. Finally, by putting $x=0$
in (\ref{seconda trasformata}) we get
\begin{align*}
\tilde{l}_*(0,\eta _1, \eta _2) = \frac{T_1(\eta _1)+T_2(\eta _2)-S(\eta _1,
\eta _2)}{\eta _1\eta _2};
\end{align*}
by taking into account (\ref{trasformata coda misura di levy}), the inverse
Laplace transform gives the boundary condition (\ref{condizione iniziale 1d}%
).
\end{proof}

\begin{os}
The boundary condition (\ref{condizione iniziale 1d}) means that,
heuristically speaking, $H_1$ and $H_2$ perform initial jumps respectively
greater than $t_1$ and $t_2$ if and only if the first hitting times of the
levels $t_1$ and $t_2$ are both equal to 0.
\end{os}

\begin{os}
We here discuss the well-posedness of the boundary value problems in Thm. %
\ref{teorema equazione governante}. For the region $x_2>x_1>0$, by applying
the time-Laplace transform to (\ref{equazione pura}), one has
\begin{align*}
\biggl (\frac{\partial}{\partial x_1 } +\frac{\partial}{\partial x_2} %
\biggr) \tilde{l}(x_1, x_2, \eta_1,\eta_2) = -S(\eta _1, \eta _2) \tilde{l}%
(x_1, x_2, \eta_1,\eta_2)
\end{align*}
which is a damped wave equation in the variables $x_1$ and $x_2$, under the
boundary condition
\begin{align*}
\tilde{l}(0, x_2, \eta_1, \eta_2)= \frac{1}{\eta _1\eta _2}T_2(\eta _2)
(S(\eta _1, \eta _2)-T_2(\eta _2)e^{-x_2T_2(\eta _2)}.
\end{align*}
Such a problem can be solved by the method of characteristics which gives
the desired solution (\ref{uvz}). Indeed the characteristic lines are $%
x_2=x_1+k$, with $k>0$ which carry the boundary condition to the region $%
x_2>x_1>0$, giving the solution
\begin{align*}
\tilde{l}(x_1, x_2, \eta_1,\eta_2)= \tilde{l}(0, x_2-x_1, \eta_1,\eta_2)
e^{-S(\eta _1, \eta _2) x_1}
\end{align*}
which coincides with (\ref{uvz}). In the same way, the characteristic lines $%
x_2=x_1+k$, with $k<0$, carry the boundary condition $\tilde{l}(x_1,0, \eta
_1, \eta _2)$ to the region $x_1>x_2>0$ and lead to the solution.

Finally, by passing again to Laplace transform in the time variables, it is
straightforward to check the well-posedness of the boundary value problem
for $l_*(x,t_1, t_2)$, as $\tilde{l}_* (x,\eta _1, \eta _2)$ is governed by
an ordinary differential equation in the variable $x$.
\end{os}

\section{Time-change of bivariate Markov processes}

Let $(X_1 (t), X_2(t))$ be a $\mathbb{R}^2$-valued Markov process such that $%
X_1 (t)$ and $X_2(t)$ are independent. Without loss of generality, we assume
that $(X_1 (0), X_2(0))=(0,0)$ almost surely. Let $p_i(x_i, t), i=1,2,$ be
the marginal density of $X_i(t)$, and let $p_*
(x_1,x_2,t)=p_1(x_1,t)p_2(x_2,t)$ denote the joint density. The following
forward equations hold:
\begin{align}
\frac{\partial}{\partial t}p_i(x_i,t) &=\mathcal{G}_i p_i(x_i,t) \qquad i=1,2
\label{equazione markov 1} \\
\frac{\partial}{\partial t} p _* (x_1, x_2, t)&= (\mathcal{G}_1+\mathcal{G}%
_2) p _* (x_1, x_2, t),  \label{equazione markov 2}
\end{align}
where the operators $\mathcal{G}_1$ and $\mathcal{G}_2$ are the duals to the
generators of $X_1$ and $X_2$. It is straightforward to verify that the
bi-parameter process $(X_1(t_1), X_2(t_2))$ has a density $p(x_1, x_2, t_1,
t_2)= p_1(x_1,t_1)p_2(x_2, t_2)$ satisfying
\begin{align}
\biggl (\frac{\partial}{\partial t_1}+\frac{\partial}{\partial t_2}\biggr ) %
p(x_1, x_2, t_1, t_2)& = (\mathcal{G}_1+\mathcal{G}_2)p(x_1, x_2, t_1, t_2).
\label{equazione Markov 3}
\end{align}

We are now interested in the time-changed process
\begin{align}
\{\bigl (X_1(L_1(t_1)), X_2(L_2(t_2)) \bigr), t_1\geq 0, t_2\geq 0 \},
\label{processo finale}
\end{align}
where $(L_1, L_2)$ is the inverse of a bivariate subordinator, which is
independent on $(X_1, X_2)$. By a simple conditioning argument, the density
of the random vector $\bigl (X_1(L_1(t_1)), X_2(L_2(t_2)) \bigr)$ has the
form
\begin{align}
h(x_1,x_2,t_1, t_2)= & \int _0^\infty \int _0^\infty
p(x_1,x_2,u,v)l(u,v,t_1, t_2)dudv  \notag \\
&+\int _0^\infty p_*(x_1, x_2, u)l_*(u,t_1, t_2)du,
\label{densita subordinato}
\end{align}
where the functions $l$ and $l_*$ have been defined in (\ref{LLL}) and (\ref%
{LLLL}). The following theorem provides a governing equation for (\ref%
{densita subordinato}), which turns out to be a generalization of eq. (\ref%
{equazione Markov 3}) as the operator $\frac{\partial}{\partial t_1}+\frac{%
\partial}{\partial t_2}$ on the left hand side is replaced by $\mathcal{D}%
_{t_1, t_2}$ defined in (\ref{Marchaud}).

\begin{te}
\label{equazione per il processo time-changed } Let $(X_1 (t), X_2(t))$ be a
$\mathbb{R}^2$-valued Markov process, with independent marginal components,
such that $(X_1 (0), X_2(0))=(0,0)$ almost surely. Moreover, let (\ref%
{equazione markov 1}) and (\ref{equazione markov 2}) hold. Then, for any $%
t_1, t_2>0$, the density (\ref{densita subordinato}) of the time-changed
process $\bigl (X_1(L_1(t_1)), X_2(L_2(t_2) \bigr)$ satisfies the following
equation
\begin{align}
\mathcal{D}_{t_1, t_2} h (x_1, x_2, t_1, t_2) = (\mathcal{G}_1+\mathcal{G}%
_2) h (x_1, x_2, t_1, t_2) \qquad x_1\neq 0, x_2 \neq 0.
\end{align}
\end{te}

\begin{proof}
By applying $\mathcal{D}_{t_1, t_2}$ to both sides of (\ref{densita
subordinato}) and using Thm. \ref{teorema equazione governante} we have
\begin{align}
& \mathcal{D}_{t_1, t_2}h(x_1, x_2,t_1, t_2)  \notag \\
&=- \int _0^\infty \int _0^\infty p(x_1, x_2,u,v) \frac{\partial}{\partial u}
l(u,v,t_1, t_2)dudv  \notag \\
&- \int _0^\infty \int _0^\infty p(x_1, x_2,u,v) \frac{\partial}{\partial v}
l(u,v,t_1, t_2)dudv  \notag \\
& - \int _0^\infty p_*(x_1, x_2, u)\frac{\partial}{\partial u}l_*(u,t_1,
t_2) du .
\end{align}
Now, we integrate by parts and use (\ref{equazione markov 1}) and (\ref%
{equazione markov 2}), together with the fact that $(X_1(0), X_2(0))=(0,0)$
almost surely; thus we get
\begin{align*}
& \mathcal{D}_{t_1, t_2}h(x_1, x_2,t_1, t_2) \\
&= \mathcal{G}_1 \int _0^\infty \int _0^\infty p(x_1, x_2,u,v) l(u,v,t_1,
t_2)dudv +\delta (x_1) \int _0^\infty p_2(x_2, v) l(0,v,t_1t_2)dv + \\
&+ \mathcal{G}_2 \int _0^\infty \int _0^\infty p(x_1, x_2,u,v) l(u,v,t_1,
t_2)dudv +\delta (x_2) \int _0^\infty p_1(x_1, u) l(u,0,t_1t_2) du \\
& + (\mathcal{G}_1+\mathcal{G}_2) \int _0^\infty p_*(x_1, x_2, u)l_*(u,t_1,
t_2) du +\delta (x_1)\delta (x_2) \overline{\phi}(t_1, t_2)
\end{align*}
where $\overline{\phi}(t_1, t_2)$ has been defined in (\ref{coda misura di
levy 2d}). In the region $x_1\neq 0, x_2\neq 0$ we have
\begin{align*}
& \mathcal{D}_{t_1, t_2}h(x_1, x_2,t_1, t_2) \\
=& (\mathcal{G}_1 +\mathcal{G}_2) \int _0^\infty \int _0^\infty p(x_1,
x_2,u,v) l(u,v,t_1, t_2)dudv \\
& + (\mathcal{G}_1+\mathcal{G}_2) \int _0^\infty p_*(x_1, x_2, u) l_*(u,t_1,
t_2) du \\
& = (\mathcal{G}_1 +\mathcal{G}_2) h(x_1, x_2, t_1, t_2),
\end{align*}
which concludes the proof.
\end{proof}

\subsection{Time-change of continuous-time Markov chains.}

We specialize the framework of this section by considering the case where $%
X_1 (t)$ and $X_2(t)$ are independent continuous-time Markov chains such
that $(X_1(0), X_2(0))=(0,0)$ almost surely. In particular, they are defined
as
\begin{align}
X_l(t)=X_l^{(n)} \qquad T ^{(n)}_l \leq t < T^{(n+1)}_l \qquad l=1,2,
\end{align}
where $X_1^{(n)}$ and $X_2^{(n)}$, $n\in \mathbb{N}_0$, denote the embedded
Markov chains, with values on the discrete spaces $\mathcal{S}_1 \subset
\mathbb{R}$ and $\mathcal{S}_2\subset \mathbb{R}$ respectively, while all
the inter-arrival times $T^{(n+1)}_l- T ^{(n)}_l$ have exponential
distribution with mean $1/\xi _l$. The transition matrices of the embedded
chains are
\begin{align}
P(X_1^{(n+1)}=j | X_1^{(n)}=i):=A_{ij}, \qquad i,j\in \mathcal{S}_1
\label{matrice A}
\end{align}
and
\begin{align}
P(X_2^{(n+1)}=j | X_2^{(n)}=i):=B_{ij}, \qquad i,j\in \mathcal{S}_2.
\label{matrice B}
\end{align}
Let $\bigl ( L_1(t), L_2 (t) \bigr )$ be again the inverse of a bivariate
subordinator. We are interested in the time-changed process $\bigl ( X_1(
L_1(t)) , X_2( L_2 (t)) \bigr )$, with values in $\mathcal{S}_1\times
\mathcal{S}_2$, which is defined by
\begin{align*}
X_l(L_l(t))=X_l^{(n)} \qquad T ^{(n)}_l \leq L_l(t) < T^{(n+1)}_l, \qquad
l=1,2,
\end{align*}
or, equivalently, by
\begin{align*}
X_l(L_l(t))=X_l^{(n)} \qquad H_l(T ^{(n)}_l) \leq t < H_l(T^{(n+1)}_l),
\qquad l=1,2.
\end{align*}
We remark that, by Theorem \ref{equazione per il processo time-changed },
the density $p(x_1, x_2, t_1, t_2)$ of the random vector $\bigl ( %
X_1(L_1(t_1)), X_2(L_2(t_2)) \bigr )$ solves the following equation for $%
x_1, x_2\neq 0$:
\begin{align}
\mathcal{D}_{t_1t_2}p(x_1, x_2, t_1, t_2) &= \xi _1 \sum _{k\in S_1} \bigl (%
A_{k, x_1}-\delta _{k,x_1} \bigr )p(k, x_2,t_1,t_2)  \notag \\
&+\xi _2 \sum _{k\in S_2} \bigl (B_{k, x_2}-\delta _{k,x_2} \bigr )p(x_1,
k,t_1,t_2),  \label{equazione caso stepped}
\end{align}
where the matrices $A_{ij}$ and $B_{ij}$ have been defined in (\ref{matrice
A}) and (\ref{matrice B}).

Clearly, for each $l=1,2$, we have that $X_l(L_l(t))$ is a stepped
semi-Markov process with inter-arrival times $\Delta _{l,n}=
H_l(T^{(n+1)}_l) - H_l( T^{(n)}_l)$. This class of processes has been widely
studied in the literature concerning the classical time-change by inverse
subordinators (see, in particular \cite{meer toaldo} and \cite{ricciuti} for
the subordination of continuous time Markov chains). Since $H_l$ has
stationary increments, the random variables $\Delta _{l,n}$, are i.i.d.
copies of $H_l(W_l)$, where $W_l$ is exponentially distributed with mean $%
1/\xi _l$.

\begin{os}
It is important to note that the time-change by $(L_1, L_2)$ introduces
dependence between the marginal components. Indeed, on the one hand, the
Markov processes $X_1(t)$ and $X_2(t)$ are respectively characterized by the
inter-arrival times following the laws of the random variables $W_1$ and $%
W_2 $ cited above; on the other hand, the processes $X_1(L_1(t))$ and $%
X_2(L_2(t))$ exhibit new inter-arrival times respectively distribuited as $%
H_1(W_1)$ and $H_2(W_2)$ (which are dependent because of the dependence
between $H_1$ and $H_2$).

The survival functions of $H_i(W_i)$, $i=1,2$ satisfy the following relation
\begin{align}
P( H_i(W_i)>t)=\mathbb{E} e^{-\xi _i L_i(t)}.  \label{ttttt}
\end{align}
For example, in the special case where $L_i(t)$ is the inverse of a stable
subordinator with index $\alpha \in (0,1)$, the following explicit
expression holds
\begin{align}
P( H_i(W_i)>t)=\mathcal{E}_\alpha(-\xi _it^\alpha),  \label{mittagleffler}
\end{align}
where $\mathcal{E}_\alpha (x)= \sum _{k=0}^\infty \frac{x^k}{\Gamma
(1+\alpha k)}$ is the one-parameter Mittag-Leffler function.

It is interesting to note that the couple of dependent inter-arrival times $%
\bigl ( H_1(W_1), H_2(W_2) \bigr )$ has joint distribution such that
\begin{align*}
P(H_1(W_1)>t _1, H_2(W_2)>t_2)= \mathbb{E}e^{-\xi _1 L_1(t_1)-\xi _2
L_2(t_2)},
\end{align*}
which clearly generalizes (\ref{ttttt}). The proof is straightforward:
\begin{align*}
&P(H_1 (W_1)>t_1, H_2(W_2)>t_2) \\
&= \xi _1\xi_2 \int _0^\infty \int _0^\infty P(H_1(w_1)>t_1, H_2(w_2)>t_2)
e^{-\xi _1w_1}e^{-\xi_2 w_2}dw_1dw_2 \\
&= \xi _1\xi_2 \int _0^\infty \int _0^\infty P(L_1(t_1)<w_1, L_2(t_2)<w_2)
e^{-\xi _1w_1}e^{-\xi_2 w_2}dw_1dw_2 \\
&= \int _0^\infty \int _0^\infty P(L_1(t_1)\in dw_1, L_2(t_2)\in dw_2)
e^{-\xi _1w_1}e^{-\xi_2 w_2} \\
&= \mathbb{E}e^{-\xi _1 L_1(t_1)-\xi_2 L_2(t_2)}.
\end{align*}
\end{os}

\subsubsection{Bivariate fractional Poisson processes}

Let us consider the special case where $X_1(t)$ and $X_2(t)$ are Poisson
processes $N_1(t)$ and $N_2(t)$ with intensities $\xi _1$ and $\xi _2$,
while $(L_1^\alpha (t), L_2^\alpha (t))$ is the inverse of a bivariate
stable subordinator with index $\alpha \in (0,1)$ (defined in Subsection \ref%
{esempio stabile}).

Then, for each $i=1,2$, $N_i(L_i(t))$ is a fractional Poisson process, i.e.
a counting renewal process with Mittag-Leffler inter-arrival times of type (%
\ref{mittagleffler}) (see e.g. \cite{beghin orsingher 1}, \cite{beghin
orsingher 2}, \cite{gorenflo}, \cite{laskin}, \cite{meer nane} and also \cite%
{beghinricciuti}, \cite{leonenko}, \cite{vell} for recent extensions). The
discrete density $p_i(k,t)= P(N_i(L_i(t))=k)$ is known to solve
\begin{align}
\frac{d^\alpha}{dt^\alpha} p_i(k,t) = -\xi _i p_i(k,t)+\xi _i p_{i}(k-1,t),
\qquad p_i(k,0)= \delta _0(k) \qquad k\in \mathbb{N}_0,  \label{poisson}
\end{align}
where the fractional derivative has been defined in (\ref{Caputo}).

Now, the process $\bigl (N_1(L_1^\alpha (t)), N_2(L_2^\alpha (t) \bigr )$ is
a bivariate renewal counting process in the sense of \cite{hunter} and turns
out to be an interesting bivariate generalization of the fractional Poisson
process (other multivariate extensions had been recently proposed in \cite%
{beghin macci 1} and \cite{beghin macci 2}). In this case, equation (\ref%
{equazione caso stepped}) has the form
\begin{align}
\mathcal{D}_{t_1t_2}^\alpha p(k_1, k_2, t_1, t_2) &= -\xi _1p(k_1, k_2, t_1,
t_2)+ \xi _1 p(k_1-1, k_2, t_1, t_2)  \label{nuova poisson} \\
& -\xi_2 p(k_1, k_2, t_1, t_2) + \xi _2p(k_1, k_2-1, t_1, t_2), \qquad k_1,
k_2 \in \mathbb{N}_0,  \notag
\end{align}
under the initial condition $p(k_1, k_2, 0,0)= \delta _0(k_1)\delta _0(k_2)$%
. Eq (\ref{nuova poisson}) is a generalization of (\ref{poisson}).

\section{CTRW limits and anomalous diffusion in anisotropic media}

Let $\bigr (B_1(t), B_2(t) \bigl )$ be a standard bi-dimensional Brownian
motion and let $L^\alpha (t)$ be the inverse of an independent univariate
stable subordinator of index $\alpha \in (0,1)$. The time-changed process
\begin{align}
\{ \bigr (B_1(L^\alpha (t)), B_2(L ^\alpha (t)) \bigl ), t\geq 0 \}
\label{isotropic}
\end{align}
has great importance in statistical physics, since it models the so-called
anomalous diffusion (see, for example, \cite{magdziarz 3}, \cite{magdziarz},
\cite{magdziarz 2}, \cite{metzler} and also \cite{d'ovidio} for recent
developments). The term "anomalous" refers to the fact that the moving
particle is subject to a sort of trapping effect which, in some sense,
delays the time with respect to what happens for $\bigr (B_1(t), B_2(t) %
\bigl )$. Indeed, the mean square displacement grows as $t^\alpha$, i.e.
slower with respect to the Brownian motion. Such a process is known to have
a density $p(x_1, x_2, t)$ solving the equation
\begin{align}
\frac{\partial ^\alpha}{\partial t^\alpha } p(x_1, x_2, t)= \frac{1}{2}
\Delta p(x_1, x_2, t ),  \label{diffusione anomala classica}
\end{align}
where the operator on the left-hand side is the Caputo derivative (see (\ref%
{Caputo})).

It is important to note that the above model assumes that the trapping
effect is the same along both the coordinate directions, i.e. the
time-changed process is isotropic as well as $\bigr (B_1(t), B_2(t) \bigl )$%
. Indeed, the time Laplace transform of the characteristic function of (\ref%
{isotropic}) has the form
\begin{align*}
\int _0^{\infty} e^{-\eta t} \mathbb{E}[e^{i \bigl (\xi _1 B_1(L^\alpha(t))
+\xi_2 B_2(L^\alpha(t)) \bigr)}] dt = \frac{\eta ^{\alpha -1}}{\eta ^\alpha+
(\xi _1^2+ \xi _2^2)/2}, \qquad \xi_1 , \xi _2 \in\mathbb{R}, \eta \geq 0
\end{align*}
which is rotationally invariant in the variables $\xi _1, \xi _2$.

Here we present a model of motion in an anisotropic medium, such that the
trapping effect depends on the coordinate direction. Let $\bigr (B_1(t),
B_2(t) \bigl )$ be a standard bi-dimensional Brownian motion and let $\bigl (%
L_1^\alpha (t), L_2^\alpha (t)\bigr)$ be the inverse of a bivariate stable
subordinator with L\'evy measure $\phi_{\alpha}$ defined in (\ref{misura di
levy caso stabile}); we construct the time-changed process
\begin{align}
\{\bigr (B_1(L_1^\alpha (t)), B_2(L_2 ^\alpha(t)) \bigl ), t\geq 0\} .
\label{anisotropic}
\end{align}
We remark that the trajectories of the models (\ref{isotropic}) and (\ref%
{anisotropic}) have a different behaviour. In fact, in the isotropic case (%
\ref{isotropic}) the test particle stays at rest simultaneously in both
directions at the intervals in which the random time $L^\alpha(t)$ has a
plateaux, while in (\ref{anisotropic}) the plateaux of $L_1^\alpha$ and $%
L_2^\alpha$ are not in general synchronized. For this reason we call this
model anisotropic sub-diffusion. In conclusion, the motion for the second
model (\ref{anisotropic}) is given by the alternation of four different
phases:

i) The particle is at rest in the $x$ direction but moves in the $y$
direction (when only $L_1^\alpha$ has a plateaux).

ii) The particle is at rest in the $y$ direction but moves in the $x$
direction (when only $L_2^\alpha$ has a plateaux).

iii) The particle stays at rest in both direction (when both $L_1^\alpha$
and $L_2^\alpha$ have a plateaux).

iv) The particle moves in both directions (otherwise).

\begin{os}
Our model of anisotropic sub-diffusion is somewhat related to recent models
of inhomogenous subdiffusion, which is governed by a time fractional
equation of space-depending order
\begin{align*}
\frac{\partial ^ {\alpha (x)} }{\partial t^ {\alpha (x)}} p(x,t)=\frac{1}{2}
\Delta p(x,t) \qquad x\in \mathbb{R}^d.
\end{align*}
Such equation has been recently studied in \cite{kian1}, \cite{kian2}, \cite%
{ricciuti} and \cite{mladen}. In this model the intensity of the trapping
effect depends on the space position rather than on the direction. In
particular, in \cite{ricciuti} and \cite{mladen} the inhomogeneous
subdiffusion is constructed by means of a particular time-change such that
the inverse subordinator and the Brownian motion exhibit a suitable
stochastic dependence.
\end{os}

According to Theorem \ref{equazione per il processo time-changed }, the
random vector $(B_1(L_1(t_1)), B_2(L_2(t_2))$ has a density $p(x_1, x_2,
t_1, t_2)$ which solves the equation
\begin{align*}
\mathcal{D}_{t_1, t_2}^\alpha p(x_1, x_2, , t_1, t_2)=\frac{1}{2} \Delta
p(x_1, x_2, , t_1, t_2) \qquad (x_1, x_2)\in \mathbb{R}^2\setminus \{(0,0)\};
\end{align*}
this is a generalization of (\ref{diffusione anomala classica}), where the
operator on the left-hand side is defined in (\ref{pseudo}). The
time-Laplace transform of the characteristic function can be computed by
using (\ref{laplace inverso}) and a simple conditioning argument:
\begin{align*}
& \int _0^\infty \int _0^\infty e^{-\eta _1t_1-\eta _2t_2}\mathbb{E}[e^ {i
(\xi _1 B_1(L_1(t_1)) + \xi _2 B_2 (L_2(t_2) ) ) }] dt_1dt_2 \\
&= \int _0^\infty \int _0^\infty e^{-\eta _1t_1-\eta _2t_2} \mathbb{E} [e^{ -%
\frac{1}{2} \xi _1^2 L_1(t_1) -\frac{1}{2} \xi _2^2 L_2(t_2)}] dt_1dt_2 \\
=&\frac{\eta _1^\alpha \eta _2^\alpha}{\eta _{1}\eta _{2}\left[ \xi _{1}^2
/2+\eta _1^\alpha \right] \left[ \xi _{2}^2/2+\eta _2^\alpha \right] }+ \\
&+\frac{ \xi _{1}^2 \xi _{2}^2\left[\eta _1^\alpha+\eta
_2^\alpha-S_{\alpha}(\eta _{1},\eta _{1})\right] }{4\eta _{1}\eta _{2}\left[
\xi _{1}^2 /2+\eta _1^\alpha\right] \left[ \xi _{2}^2/2+\eta _2^\alpha\right]
\left[ \xi _{1}^2 /2+ \xi _{2}^2/2+S_{\alpha}(\eta _{1},\eta _{1})\right] },
\end{align*}
where $S_{\alpha}$ has been defined in (\ref{lap}).

\subsection{Functional limit results}

As said in the introduction, a well established theory (see e.g. \cite%
{becker}, \cite{meer libro}, \cite{meer spa}, \cite{meer jap}, \cite{meer
straka}, \cite{straka 2}) shows that (\ref{isotropic}) arises as CTRW
scaling limit. For the reader's convenience we first report the main aspects
of such a theory (in the particular case of bi-dimesional CTRW); then we
prove that, by constructing another suitable class of CTRWs, also
anisotropic subdiffusion (\ref{anisotropic}) arises as a particular scaling
limit.

Let $(Y_i^1, Y_i^2), i=1\dots, n$, be i.i.d. random vectors in $\mathbb{R}^2$
representing the jumps of a particle and let $S_n= \sum _{k=1}^n (Y^1_k,
Y_k^2)$ be the discrete-time random walk giving the location of the particle
after $n$ jumps, where $S_0=(0,0)$. Moreover, let $J_i\in \mathbb{R}^+$ be
i.i.d. random variables representing the inter-arrival times between
consecutive jumps and let $T_n= \sum _{k=1}^n J_k$ be the time of the $n$-th
jump.

Let $N(t)= \max \{n\geq 0: T_n\leq t\}$ be a renewal counting process giving
the number of jumps up to time $t\in \mathbb{R}^+$. A CTRW is defined as the
time-changed process
\begin{align}
S_{N(t)}= \sum _{k=1}^{N(t)} (Y^1_k, Y_k^2), \qquad t\in \mathbb{R}^+,
\label{sssssss}
\end{align}
namely a process with jumps $(Y^1_k, Y^2_k)$ separated by inter-arrival
times $J_k$.

We now recall some functional limit results concerning the case of infinite
mean inter-arrival times. Assume that the random variables $J_i$ belong to
the domain of attraction of a positively skewed stable law with index $%
\alpha \in (0,1)$. Then
\begin{align*}
c^{-\frac{1}{\alpha}}T_{[ct]}\to H^\alpha(t) , \qquad \text{as} \,\, c\to
\infty
\end{align*}
in the $J_1$ topology, where $H^\alpha(t)$ is a $\alpha$-stable subordinator
and $[x]$ denotes the biggest integer less than (or equal to) $x$. Moreover
\begin{align*}
c^{-\alpha}N(ct) \to L^\alpha(t), \qquad \text{as} \,\, c\to \infty
\end{align*}
in the $M_1$ topology, where $L^\alpha(t)= \inf \{ x: H^\alpha(x)>t\}$ is an
inverse $\alpha $ stable subordinator. Moreover, if the $(Y^1_i, Y^2_i)$
belong to the domain of attraction of a standard normal law, then $c^{-\frac{%
1}{2}}S_{[ct]}$ converges to a standard Brownian motion $(B_1(t), B_2(t))$
in $J_1$ topology. Then, by combining the above results,
\begin{align*}
c^{-\frac{\alpha}{2}}S_{N([ct])} \to \bigl (B_1(L^\alpha(t)),
B_2(L^\alpha(t) \bigr ) \qquad \text{as}\, \, c\to \infty
\end{align*}
in the $M_1$ topology, where the limit process is the subdiffusion defined
by (\ref{isotropic}).

In the following subsection, we show how it is possible to construct a CTRW
converging to the new process (\ref{anisotropic}).

\subsubsection{Limits of CTRWs driven by bivariate renewal processes}

Let
\begin{equation*}
(J_1^1, J_1^2), \dots, (J_n^1, J_n^2)
\end{equation*}
be i.i.d. random vectors in $\mathbb{R}^2_+$ (such that, for each $i$, $%
J_i^1 $ and $J_i^2$ are possibly dependent) and let $(T_n^1, T_n^2)$ be the
simple random walk on $\mathbb{R}_+^2$ defined by
\begin{align}
(T_n^1, T_n^2)= \biggl (\sum _{k=1}^nJ_k^1\, ,\, \sum _{k=1}^n J_k^2 \biggr) %
\qquad n\in \mathbb{N}_0.  \label{ttt}
\end{align}
Let $(N^1(t), N^2(t))$ be a bivariate renewal counting process in the sense
of \cite{hunter}, namely, for each $i=1,2$,
\begin{align}
N^i(t)= \max \{ n\in \mathbb{N}_0: T_n^i \leq t \} \qquad t\in \mathbb{R}^+.
\label{definizione bivariate renewal}
\end{align}
$N^1(t)$ and $N^2(t)$ are possibly dependent. Consider now another random
walk $(S_n^1, S_n^2)$, defined by
\begin{align}
(S_n^1, S_n^2)= \biggl (\sum _{k=1}^nY_k^1\, , \, \sum _{k=1}^n Y_k^2 %
\biggr) \qquad n\in \mathbb{N}_0,  \label{random walk nostro}
\end{align}
where $(Y_k^1, Y_k^2), k=1, \dots, n,$ are i.i.d random vectors in $\mathbb{R%
}^2$. \newline
We finally consider the CTRW defined by the bivariate time-change
\begin{align}
(S ^1 _{N ^1 (t)} , S ^2_{N^2(t)})= \biggl ( \sum _{k=1}^{N^1(t)} Y_k^1,
\sum _{k=1}^{N^2(t)} Y_k^2 \biggr )\ \qquad t\in \mathbb{R}^+.
\label{CTRW nostro}
\end{align}

We note that the marginal components of (\ref{CTRW nostro}) do not have
simultaneous jumps (unlike the ones of (\ref{sssssss})). Moreover, if
considered separately, both $S ^1 _{N ^1 (t)}$ and $S ^2_{N^2(t)}$ are CTRWs
of the type already treated in the literature and then we can refer to the
classical theory to perform their scaling limits.

The novelty of the following theorem is to present functional limit results
in the case of bivariate time-change. We will denote by $\overset{J_1}{\to}$
(resp. $\overset{M_1}{\to}$) the convergence in $J_1$ (resp. $M_1$) topology
on the space $D^2([0,\infty))$.

\begin{te}
\label{teorema sulla convergenza} Let $(J_k^1, J_k^2)$ belong to the domain
of attraction of a bivariate stable law with index $\alpha \in (0,1)$ and
support on $\mathbb{R}_2^+$. Then
\begin{align}
c^{-\frac{1}{\alpha}} (T^1_{[ct]}, T^2_{[ct]}) \overset{J_1}{\to}
(H_1^\alpha(t), H_2^\alpha (t)), \qquad \text{as}\,\, c\to \infty
\label{primo limite}
\end{align}
and
\begin{align}
\bigl (c^{-\alpha}N_1 (ct), c^{-\alpha}N_2 (c t) \bigr ) \overset{M_1}{\to}
(L_1^\alpha(t), L_2^\alpha (t)), \qquad \text{as}\,\, c\to \infty,
\label{secondo limite}
\end{align}
where $(H_1 ^\alpha(t), H_2^\alpha (t))$ is a bivariate stable subordinator
(see Subsection \ref{esempio stabile}) and $(L_1 ^\alpha(t), L_2^\alpha (t))$
is its inverse (in the sense explained in section 3). Moreover, if the
random vectors $(Y^1_i, Y^2_i)$ in (\ref{random walk nostro}) belong to the
domain of attraction of a standard normal law, we have
\begin{align}
(c^{-\frac{\alpha}{2}}S ^1 _{N ^1 ([ct])} , \, c^{-\frac{\alpha}{2}}
S^2_{N^2([ct])}) \overset{M_1}{\to} \bigl (B_1(L_1 ^\alpha(t)), B_2(L_2
^\alpha (t) \bigr) \qquad \text{as}\,\, c\to \infty  \label{s3}
\end{align}
where $(B_1(t), B_2(t) )$ denotes a standard Brownian motion.
\end{te}

\begin{proof}
Since $(J_k^1, J_k^2)$ belong to the domain of attraction of a bivariate $%
\alpha$-stable law, then, by Thm 6.21 at page 168 of \cite{meer libro}, we
have
\begin{align}
c^{-\frac{1}{\alpha}} (T^1_{[ct]}, T^2_{[ct]}) \overset{d}{\to}
(H_1^\alpha(t), H_2^\alpha(t)) \qquad \text{ for all}\,\, t\geq 0, \qquad
\text{as} \,\, c\to \infty,  \label{s1}
\end{align}
where $\overset{d}{\to}$ denotes convergence in distribution. Then, by using
\cite{skorohod}, Thm. 2.7, condition (\ref{s1}) implies (\ref{primo limite}).

The bivariate renewal process $(N^1(t), N^2(t))$ is a continuous functional
in $M_1$ topology; then the continuous mapping argument provided by Thm.
1.6.13 in \cite{silvestrov} (page 56) gives (\ref{secondo limite}).

We finally observe that the set of discontinuities of the limit processes $%
(B_1(t), B_2(t))$ and $(H_1^{\alpha}(t), H_2^{\alpha} (t))$ are obviously
disjoint as the Brownian motion has continuous sample paths; hence, another
continuous mapping argument (i.e. a vector composition of processes in the
sense of \cite{silvestrov}, Section 2.7) leads to (\ref{s3}).
\end{proof}

\begin{os}
Theorem \ref{teorema sulla convergenza} is restricted to the case where the
inter-arrival times belong to the domain of attraction of a stable law,
while the jumps belong to the domain of attraction of a normal distribution.
We trust that a more general result could be obtained. Indeed, by using
triangular array convergence, one could remove such assumptions and study
the case where a rescaled CTRW converges to a L\'evy process time-changed by
the inverse of a general bivariate subordinator. However, in this case,
there are some difficulties in treating simultaneous jumps of the components
of the limit processes. So, a further analysis is needed on these points.
\end{os}

\vspace{0.2cm}

\textbf{Acknowledgements.} The authors thank the referee for his/her careful
reading of the first version of the manuscript. Moreover they also thank
Patrizia Semeraro for useful comments on her papers cited in the
bibliography, Dmitrii Silvestrov for useful discussion on the $M_1$
convergence, and Fabrizio Durante for suggesting the references \cite%
{EsmaeiliKluppelberg2011} and \cite{EsmaeiliKluppelberg2013}.

The authors acknowledge the support of GNAMPA-INdAM. Moreover C.M.
acknowledges the support of MIUR Excellence Department Project awarded to
the Department of Mathematics, University of Rome Tor Vergata (CUP
E83C18000100006).

\end{document}